\documentclass[10pt]{amsart}
\usepackage{hyperref}
\usepackage{color}
\usepackage{amsrefs}
\usepackage[pdftex]{graphicx}
    \newtheorem{thm}{Theorem}[section]
    
    \newtheorem{lem}[thm]{Lemma}
    \newtheorem{cor}[thm]{Corollary}
    \theoremstyle{definition}

    \newtheorem{defn}[thm]{Definition}

    \theoremstyle{remark}

\newcommand{\executeiffilenewer}[3]{%
\ifnum\pdfstrcmp{\pdffilemoddate{#1}}%
{\pdffilemoddate{#2}}>0%
{\immediate\write18{#3}}\fi%
}

\newcommand{\fakeenv}{}

\newenvironment{restate}[2]                                    
{ 
 \renewcommand{\fakeenv}{#2}                              
 \theoremstyle{plain} 
 \newtheorem*{\fakeenv}{#1~\ref{#2}}                
 \begin{\fakeenv}
}
{
 \end{\fakeenv}
}

\title{A short proof of the bounded geodesic image theorem}

\author{Richard C. H. Webb}
\address{Mathematics Institute, University of Warwick, Coventry, CV4 7AL, United Kingdom.}
\email{R.C.H.Webb@warwick.ac.uk}

\begin{document}

\begin{abstract}

We give a combinatorial proof, using the hyperbolicity of the curve graphs, of the bounded geodesic image theorem of Masur--Minsky. Recently it has been shown that curve graphs are uniformly hyperbolic, thus a universal bound can be given for the diameter of the geodesic image. We also generalize the theorem for projections to markings of the whole surface.

\end{abstract}

\maketitle

\section{Introduction}

We write $S_{g,p}$ to denote the genus $g$ surface with $p$ points removed and $\xi(S)=3g-3+p$ to denote the \textit{complexity} of $S=S_{g,p}$. We say a simple closed curve on $S$ is \textit{essential} if it does not bound a disc or once-punctured disc. In general, we say that an isotopy class of some subset of $S$ \textit{misses} another isotopy class of some subset if they admit disjoint representatives, and otherwise we say that they \textit{cut}. A \textit{curve} is an isotopy class of essential simple closed curve. We write $\mathcal{C}(S)$ to denote the \textit{curve graph} of $S$, whose vertex set is the set of curves on $S$ with edges between non-equal curves that miss; this is the 1-skeleton of the \textit{curve complex} which was introduced by Harvey \cite{Harvey}. Throughout, $S=S_{g,p}$ with $\xi(S)\geq 2$. For the surfaces $S_{0,4}$ and $S_{1,2}$, one can use the Farey graph, the description of its geodesics and a lifting argument to prove Theorem \ref{bgit}.

We shall abuse notation by simply writing $\gamma$ to mean both the simple closed curve $\gamma$ and its isotopy class. We write $d_S$ to denote the path metric on $\mathcal{C}(S)$ with unit length edges. A sequence of curves $g=(\gamma_i)$ is a \textit{geodesic} if for all $i\neq j$, we have $d_S(\gamma_i,\gamma_j)=|i-j|$. We say $\mathcal{C}(S)$ is $\delta$-hyperbolic if for all geodesic triangles $g_1,g_2,g_3$, we have $g_1\subset N_\delta(g_2\cup g_3)$, where $N_\delta$ is the metric closed $\delta$-neighbourhood.

\begin{thm}[\cite{MasurMinsky99}]
\label{thm:hyp}

Fix $S=S_{g,p}$ with $\xi(S)\geq 2$. There exists $\delta\geq 0$ such that $\mathcal{C}(S)$ is $\delta$-hyperbolic.

\end{thm}

We write \textit{subsurface} to denote a compact, connected, proper subsurface of $S$ such that each component of its boundary is essential in $S$. Throughout, we do not consider subsurfaces that are homotopy equivalent to $S_{0,3}$; in this case Theorem \ref{bgit} is straightforward.

For a non-annular subsurface $Y\subset S$. We write $\partial Y$ for the boundary of $Y$. We now define a map $\pi_Y : \mathcal{C}_0(S)\rightarrow \mathcal{P}(\mathcal{AC}_0(Y))$, where $\mathcal{AC}(Y)$ is the \textit{arc and curve complex} of $Y$, and generally $\mathcal{P}(X)$ is the set of subsets of $X$. Given a curve $\gamma\in\mathcal{C}(S)$, isotope $\gamma$ so that it intersects $Y$ minimally. We define $\pi_Y(\gamma)$ to be the arcs and/or curves $\gamma\cap Y\subset Y$. This is non-empty if and only if $\gamma$ cuts $Y$. The map $\pi_Y$ is the \textit{subsurface projection} to the arc and curve complex of $Y$. We write $\pi_Y(A)=\cup_{\gamma \in A}\pi_Y(\gamma)$.

When $Y$ is an annulus we write $\partial Y$ for the core curve of $Y$. This core curve represents a subgroup of $\pi_1(S)$ and therefore there is an associated cover $p_Y:S_Y\rightarrow S$, where $S_Y$ is homeomorphic to the interior of an annulus. There is a homeomorphic lift of $Y$ to $S_Y$ which we write $Y'$. One can compactify $S_Y$ to a closed annulus by using a hyperbolic metric on $S$. Let $\mathcal{AC}_0(Y)$ be the set of arcs that connect one boundary component of $S_Y$ to the other, modulo isotopies that fix the endpoints. Two arcs are adjacent if they admit disjoint representatives. We write $\mathcal{AC}(Y)$ to denote this graph. Given a curve $\gamma$ that cuts $Y$, we define $\pi_Y(\gamma)$ to be the set of arcs of the preimage $\tilde{\gamma}=p^{-1}_Y\gamma$ that connect the two boundary components of $S_Y$. Otherwise, $\pi_Y(\gamma)=\emptyset$. This defines the subsurface projection $\pi_Y:\mathcal{C}_0(S)\rightarrow\mathcal{P}(\mathcal{AC}_0(Y))$ when $Y$ is an annulus.

We write $d_{\mathcal{AC}(Y)}$ to denote the standard metric on the graph $\mathcal{AC}(Y)$. We write $d_Y(A)=\textnormal{diam}_{\mathcal{AC}(Y)}(\pi_Y A)$ and $d_Y(A,B)=\textnormal{diam}_{\mathcal{AC}(Y)}(\pi_Y(A)\cup \pi_Y(B))$. The following lemma is immediate, see also \cite[Lemma 2.2]{MasurMinsky00}.

\begin{lem}
\label{bp}

Let $Y$ be a subsurface of $S$ and let $\gamma_1,\gamma_2$ be curves on $S$. Suppose that $\gamma_1$ cuts $Y$, $\gamma_2$ cuts $Y$ and $\gamma_1$ misses $\gamma_2$. Then $d_Y(\gamma_1,\gamma_2)\leq 1$. \hfill $\square$

\end{lem}

We shall give a proof of the bounded geodesic image theorem of Masur--Minsky \cite[Theorem 3.1]{MasurMinsky00}. We shall give a bound that depends only on $\delta$, where $\mathcal{C}(S)$ is $\delta$-hyperbolic.

\begin{restate}{Theorem}{bgit}

Given a surface $S$ there exists $M=M(\delta)$ such that whenever $Y$ is a subsurface and $g=(\gamma_i)$ is a geodesic such that $\gamma_i$ cuts $Y$ for all $i$, then $d_Y(g)\leq M$.

\end{restate}

Recently, it has been shown that there exists $\delta$ such that $\mathcal{C}(S)$ is $\delta$-hyperbolic for all surfaces $S$ in Theorem \ref{thm:hyp},  see Aougab \cite{Aougab13}, Bowditch \cite{Bow13}, Clay--Rafi--Schleimer \cite{ClayRafiSchleimer13} and Hensel--Przytycki--Webb \cite{HenselPrzytyckiWebb13}.

\begin{cor}

There exists $M$ independent of the surface $S$ in Theorem \ref{bgit}.

\end{cor}

In the last section, we describe markings on $S$ in terms of graphs embedded in $S$ that fill. Given a multicurve $\alpha$, and a curve $\gamma$ that fills with $\alpha$, one can define such a graph $\Gamma_\alpha(\gamma)$. This gives a projection $\Gamma_\alpha$ to a set of markings. Our proof of Theorem \ref{bgit} generalizes to these projections.

\begin{restate}{Theorem}{bm}

Suppose a multicurve $\alpha$ and a geodesic $g=(\gamma_i)$ satisfy $\gamma_i,\alpha$ fill $S$ for each $i$. Then $\textnormal{diam}_{\mathcal{M}_{k,l}(S)}(\Gamma_\alpha(g))\leq M$, where $M$ depends on $S$.

\end{restate}

\section{Loops and surgery}

\subsection{Loops}
Throughout this section, $\alpha$ and $\beta$ are both collections of pairwise disjoint, essential, simple closed curves on $S$ such that $\alpha$ and $\beta$ intersect minimally (equivalent to $\alpha$ and $\beta$ do not share a bigon, see for example \cite[Proposition 1.7]{Farb Margalit}) and $\alpha,\beta$ fill $S$.

We say a collection of simple closed curves $\{\gamma_i\}$ is \textit{sensible} if they are essential, pairwise in minimal position, and with no triple points, i.e. for distinct $i,j,k$, we have $\gamma_i\cap\gamma_j\cap\gamma_k=\emptyset$.

Let $\gamma,\alpha,\beta$ be sensible. Recall that whenever we orient $\gamma$ and $\beta$ arbitrarily, each point $\gamma\cap\beta$ has a sign of intersection $\pm 1$. We say a pair of such points \textit{have opposite sign} if the signs of intersection are non-equal, and \textit{have same sign} otherwise. This notion does not depend on the orientation of $\gamma,\beta$ or $S$.

\begin{defn}\label{circuit}We say that $\gamma$ is an $(\alpha,\beta)$\textit{-loop} if for each arc $b\subset\beta-\alpha$ we have $|\gamma\cap b|\leq 2$ with equality only if $\gamma\cap \beta$ have opposite sign.\end{defn}

Definition \ref{circuit} is inspired by Leasure's $(\alpha\cup\beta)$\textit{-cycles} \cite[Definition 3.1.6]{LeasureThesis}. These cycles allow one to construct quasigeodesics on closed surfaces with a combinatorial description. Definition \ref{circuit} is an adaptation, which allows one to work on punctured surfaces. Both $(\alpha,\beta)$-loops and Leasure's cycles satisfy some variant of Lemma \ref{cirproj}, however cycles a priori require larger constants for Lemma \ref{cirproj} and a more careful proof since they are not necessarily in minimal position with $\alpha$ and $\beta$. Our surgery argument to construct $(\alpha,\beta)$-loops from curves is necessarily more technical, but they will intersect $\alpha$ and $\beta$ minimally.

\subsection{Surgery}

Suppose that $\gamma,\alpha,\beta$ are sensible. We shall describe a surgery process on $\gamma$ to construct an $(\alpha,\beta)$-loop which will be written $\gamma'$. If $\gamma$ is an $(\alpha,\beta)$-loop then we set $\gamma'=\gamma$. If $\gamma$ is not an $(\alpha,\beta)$-loop then let $c$ be a minimal (with respect to inclusion) connected subarc $c\subset\gamma$ such that there exists an arc $b\subset \beta-\alpha$ with either

\begin{itemize}
\item $c\cap b$ is a pair of points with same sign
\item $c\cap b$ has cardinality at least 3
\end{itemize}

Since $c$ is minimal we have that $c$ has endpoints on $b$, $b$ is the unique arc with properties described above, and $|c\cap b|\leq 3$. Thus, each arc $b'\subset\beta-\alpha$ such that $b'\neq b$, we have $|c\cap b'|\leq 2$ with equality only if $c\cap b'$ have opposite sign.

In what follows, we write $N=N(\beta)$ to denote a closed regular neighbourhood of $\beta$. We now describe how to construct $\gamma'$, in each case of how $c$ intersects $b$.

\begin{figure}
\executeiffilenewer{samesign.svg}{samesign.pdf}%
{inkscape -z -D --file=samesign.svg %
--export-pdf=samesign.pdf --export-latex}%
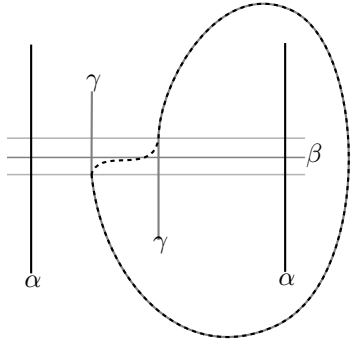%

\caption{The curve $\gamma'$ is dotted.}
\label{fig:samesign}
\end{figure}

\begin{itemize}

\item [Case 1:] $|c\cap b|=2$ and $c\cap b$ have same sign. See Figure \ref{fig:samesign}. Write $R\subset N-\alpha$ to denote the rectangle with $b\subset R$. Let $\{p_1,p_2\}=c\cap\partial R$. Connect $p_1$ to $p_2$ by an arc $a\subset R$ that intersects $b$ once and intersects $c$ only at the endpoints of $a$. We let $\gamma'$ be the simple closed curve $a\cup (c-R)$.

\begin{figure}
\executeiffilenewer{other.svg}{other.pdf}%
{inkscape -z -D --file=other.svg %
--export-pdf=other.pdf --export-latex}%
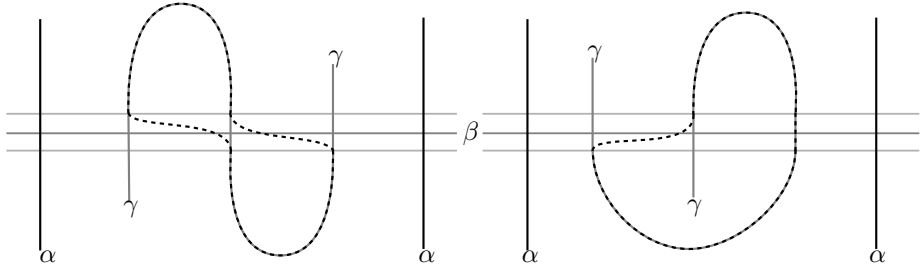%

\caption{Surgery in Case 2 on the left, and Case 3 on the right.}
\label{fig:other}
\end{figure}

\item [Case 2:] $|c \cap b|=3$ with alternating signs of intersection with respect to some order on $b$. See Figure \ref{fig:other}. Let $p_1,p_2,p_3$ be the points $c\cap b$ in some order along $b$. Let $c_1,c_2\subset c$ be arcs such that $c_1\cup c_2=c$, $\partial c_1=\{p_1,p_2\}$ and $\partial c_2=\{p_2,p_3\}$. Connect $c_1\cap\partial R$ to $c_2\cap\partial R$ by two disjoint arcs $a_1,a_2\subset R$ so that $a_1$ intersects $c_1,c_2$ only at its endpoints and intersects $b$ once, and similarly $a_2$. We let $\gamma'=a_1\cup (c_1-R)\cup a_2 \cup (c_2-R)$.

\item [Case 3:] $|c\cap b|=3$ with non-alternating signs of intersection. See Figure \ref{fig:other}. We define $\gamma'$ in a similar fashion as Case 1.

\end{itemize}

\begin{lem}
\label{minpos}

In each case above, $\gamma'$ is in minimal position with $\beta$ and with $\alpha$. Furthermore, $\gamma'$ is essential and an $(\alpha,\beta)$-loop.

\end{lem}

\begin{proof}

Any arc of $\gamma'-\beta$ is isotopic in $S-\beta$ to some arc of $\gamma-\beta$. Therefore, $\gamma'$ and $\beta$ cannot share a bigon since $\gamma$ and $\beta$ do not, thus $\gamma'$ is essential. We now show that $\gamma'$ and $\alpha$ do not share a bigon in all the cases of the surgery process described above, via contradiction.

\begin{itemize}

\item [Case 1:] See Figure \ref{fig:bigons}. Pick an innermost bigon $B$ between the pair $\gamma',\alpha$. We must have the arc $a\subset \partial B$, otherwise $\gamma$ and $\alpha$ share a bigon. We have $a\cap b\neq\emptyset$, and since $\gamma'$ and $\beta$ do not share a bigon, we must have one endpoint of $b$ in $\partial B$. Let $\{p\}=B\cap c\cap b$, and $c_\alpha\subset \gamma-\alpha$ be the arc with $p\in c_\alpha$. Now $\gamma$ and $\beta$ do not share a bigon, and neither does $\gamma$ intersect itself, thus $c_\alpha$ is contained with the disc $B\cup T$, where $T$ is the triangle region adjacent to $B$ cobounded by $\gamma',\gamma$ and $\beta$. We conclude that $c_\alpha$ and $\alpha$ cobound a bigon, contradicting $\gamma$ and $\alpha$ do not share a bigon.

\item [Case 2:] It suffices to show each connected component of $S-(\gamma'\cup\alpha)$ adjacent to at least one of the arcs $ a_1,a_2$ is not a bigon. We start with the component containing $p_2$: if this is a bigon $B$, then the arc $b'\subset b-\gamma'$ with $p_2\in b'$ satisfies $b'\subset B$ hence $b'$ cobounds a bigon with $\gamma'$. This contradicts $\gamma'$ and $\beta$ do not share a bigon. Now we argue that the component containing $p_1$ is not a bigon (and similarly $p_3$). Suppose this component was a bigon $B$, first suppose that $a_1\subset\partial B$ but $a_2\cap \partial B=\emptyset$, then one can follow a similar argument as in Case 1. If $a_1,a_2\subset \partial B$, then see Figure \ref{fig:bigons} on the right. A similar argument again can be given as in Case 1.

\item [Case 3:] One can argue similarly to that of Case 1.\qedhere \end{itemize}\end{proof}

\begin{figure}
\executeiffilenewer{bigons.svg}{bigons.pdf}%
{inkscape -z -D --file=bigons.svg %
--export-pdf=bigons.pdf --export-latex}%
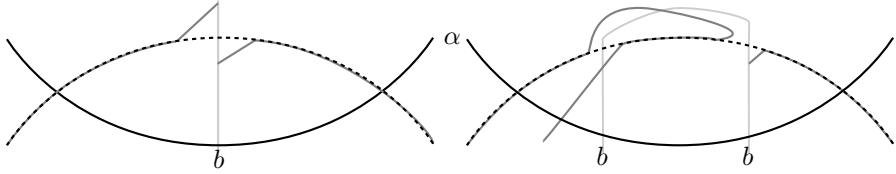%

\caption{The argument in Lemma \ref{minpos}}
\label{fig:bigons}
\end{figure}

Our Lemma \ref{cirpath} is the generalization of \cite[Proposition 3.1.7]{LeasureThesis}.

\begin{lem}
\label{cirpath}

Suppose $\gamma_1,\gamma_2,\alpha,\beta$ are sensible and $\gamma_1$ misses $\gamma_2$. Then the $(\alpha,\beta)$-loops $\gamma_1',\gamma_2'$ constructed by the surgery method above satisfy $i(  \gamma_1' , \gamma_2'  )\leq 4$. Furthermore, if $\gamma_1$ misses $\alpha$, or $\beta$, then $\gamma_1'$ misses $\alpha$, or $\beta$, respectively. \hfill $\square$

\end{lem}

\begin{lem}
\label{cirquasi}

Let $\alpha'$ be a component of a multicurve $\alpha$ on $S$ and let $\beta$ be a curve on $S$. Suppose $\alpha',\beta$ fill $S$. Then there exists a $(4,0)$-quasigeodesic $\alpha'=\gamma_0,\gamma_1,...,\gamma_n=\beta$ with $\gamma_i$ a $(\alpha,\beta)$-loop for every $0<i<n$.

\end{lem}

\begin{proof}

Start with a geodesic $\gamma_0,...,\gamma_m$ of curves from $\alpha'$ to $\beta$, such that this collection of curves is sensible. Using the surgery process on each $\gamma_i$ with $1\leq i \leq n-1$, we obtain a sequence $\gamma_1',...,\gamma_{m-1}'$ of $(\alpha,\beta)$-loops. We have $i(\gamma_i',\gamma_{i+1}')\leq 4$ for each $i$ by Lemma \ref{cirpath}, therefore $d_S(\gamma_i',\gamma_{i+1}')\leq 4$ and $d_S(\gamma_i',\gamma_j')\leq 4|i-j|$ for each $i,j$. If for some $i>j$ we have $i-j > d_S(\gamma_i',\gamma_j')$, then we connect $\gamma_i'$ and $\gamma_j'$ with a geodesic and surger each vertex of it using $\alpha,\beta$ again. Repeating this process, we obtain the required quasigeodesic of $(\alpha,\beta)$-loops.\end{proof}

We remark that if $S\neq S_{1.2}$ then in Lemma \ref{cirquasi} we can take a $(3,0)$-quasigeodesic, and for all but finitely many surfaces we can take a $(2,0)$-quasigeodesic.

\begin{lem}
\label{cirproj}
Let $Y$ be a subsurface of $S$ and suppose $\partial Y$ and $\beta$ fill $S$. Let $\gamma$ be a $(\partial Y, \beta)$-loop that cuts $\partial Y$. Then $d_Y(\gamma,\beta)\leq 2$ if $Y$ is non-annular and $d_Y(\gamma,\beta)\leq 5$ otherwise.

\end{lem}

\begin{proof}

If $Y$ is non-annular then any pair of arcs in the projection will intersect at most twice by Definition \ref{circuit}, so one can consider a closed regular neighbourhood of the arcs to prove the required bound on distance.

If $Y$ is annular then suppose for contradiction that $d_Y(\gamma,\beta)\geq 6$. Then there exist arcs $\delta^*\in \pi_Y(\gamma)$ and $\epsilon^*\in\pi_Y(\beta)$ with $|\delta^*\cap \epsilon^*|\geq 5$. Following a claim from \cite[Section 10]{MasurSchleimer13}, if we isotope the triangles cobounded by $\partial Y, \beta,\gamma$ into $Y$ (this retains minimal position), we have that $|\delta^*\cap\epsilon^*\cap Y'| \geq 3$, where $Y'$ is the homeomorphic lift of $Y$. Therefore there exists an arc of $\beta-\partial Y$ which intersects $\gamma$ at least two times with the same sign, contradicting $\gamma$ a $(\partial Y,\beta)$-loop.\end{proof}

\section{The proof}

Let $\gamma$ be a curve and $P$ be a set of curves. We say  $\gamma$ is $\epsilon$\textit{-close} to $P$ if for some curve $\beta$ of $P$ we have $d_S(\gamma,\beta)\leq \epsilon$. Throughout this section, $\delta$ is a constant such that $\mathcal{C}(S)$ is $\delta$-hyperbolic, see Theorem \ref{thm:hyp}.

\begin{lem}
\label{bg}
There exists $D=D(\delta)$ such that for any subsurface $Y$, component $\alpha\subset\partial Y$,  and geodesic  $\alpha=\gamma_0,\gamma_1,...,\gamma_n=\beta$ with $n\geq 3$, we have $d_Y(\gamma_i,\beta)\leq D$ whenever $i\geq 2$.
\end{lem}

\begin{proof}

Use Lemma \ref{cirquasi} to construct a (4,0)-quasigeodesic $Q$ of $(\partial Y,\beta)$-loops from $\alpha$ to $\beta$. For each $i$, we have $\gamma_i$ is $D'$-close to $Q$ where $D'=D'(\delta)$. For an explicit $D'$, we can take $D'=D''+2$, where $D''$ is the largest integer with $D''\leq \delta \lceil \log_2(26 D'')\rceil$. See for example \cite[Chapter III.H]{BridsonHaefliger}.  Using Lemma \ref{bp}, we can take $D=2D'+B$, where $B$ is the bound provided in Lemma \ref{cirproj}.\end{proof}

\begin{thm}
\label{bgit}

Given a surface $S$ there exists $M=M(\delta)$ such that whenever $Y$ is a subsurface and $g=(\gamma_i)$ is a geodesic such that $\gamma_i$ cuts $Y$ for all $i$, then $d_Y(g)\leq M$.

\end{thm}

\begin{proof}

Take $M=4\delta+2D+4$, where $D$ is defined as in Lemma \ref{bg}. Fix $i<j$. We shall show that $d_Y(\gamma_i,\gamma_j)\leq M$. Fix $\alpha$ a component of $\partial Y$. Let $I=N_{\delta+1}(\alpha)\cap g$. There exists $g'=(\gamma_{i'},...,\gamma_{j'})$ a geodesic of length at most $2\delta+2$ such that $I\subset g' \subset g$.

Let $P$ be a geodesic from $\alpha$ to $\gamma_i$ and $Q$ be a geodesic from $\beta$ to $\gamma_j$. Let $i''=\max \{ i, i'-1 \}$ and $j''=\min \{ j, j'+1 \}$. Since geodesic triangles are $\delta$-slim, we have either $\gamma_{i''}$ is $\delta$-close to $P$ and $\gamma_{j''}$ is $\delta$-close to $Q$, or, there exists adjacent vertices of $g-g'$ with one $\delta$-close to $P$ and the other $\delta$-close to $Q$. By lemmas \ref{bp} and \ref{bg} we have that $d_Y(\gamma_i,\gamma_j)\leq D+\delta+(2\delta+4)+\delta+D=M$.\end{proof}

We remark that $M$ need not be optimal for each surface. For example, for $S_2$ it may be better to consider Leasure's cycles, which give $(2,0)$-quasigeodesics, whereas a priori we are taking $(3,0)$-quasigeodesics in Lemma \ref{cirquasi}. Similar surgery arguments may produce better results for other surfaces. Also, for all but finitely many surfaces, in Lemma \ref{cirquasi} we can take a $(2,0)$ quasigeodesic; this improves on the constant $D$.

\section{Generalization to markings}

We thank Brian Bowditch for suggesting this generalization and set-up. Defining the markings that we wish to discuss has similarities with \cite[Section 6]{MasurMosherSchleimer}.

Given a multicurve $\alpha$ and a curve $\beta$ such that $\alpha,\beta$ fill $S$, let $B$ be a maximal collection of pairwise non-isotopic arcs of $\beta-\alpha$ in $S-\alpha$. We let $\Gamma_\alpha(\beta)$ be the graph embedded in $S$ by taking the union $\alpha\cup B$. This may not be well-defined but there is bounded intersection between two such graphs, in terms of $S$, between any pair of choices of $B$.

Since $\alpha,\beta$ fill $S$, it follows that there are no essential simple closed curves on $S$ that are disjoint from $\Gamma_\alpha(\beta)$, i.e. $\Gamma_\alpha(\beta)$ fills $S$. Furthermore, by an Euler characteristic argument, the number of edges of $\Gamma_\alpha(\beta)$ can be bounded in terms of the surface $S$. Let $k_1$ be this bound.

We write $\mathcal{M}_k(S)$ to denote the set of (isotopy classes of) embedded graphs that fill $S$ with at most $k$ edges. Let $\mathcal{M}_{k,l}(S)$ be the graph with vertex set $\mathcal{M}_k(S)$ with two vertices $G_1,G_2$ adjacent if $i(G_1,G_2)\leq l$. Here, $i(G_1,G_2)=\min|\Gamma_1\cap \Gamma_2|$ where the minimum is taken over representatives $\Gamma_i$ of the isotopy classes $G_i$, where $i=1,2$.

Let $k_2$ be a bound for the number of edges of any \textit{clean complete marking} on $S$ regarded as a graph on $S$, see \cite{MasurMinsky00} for definitions. The graph of clean complete markings on $S$ is connected. Write $l_1=\textnormal{max}_M i(M,M')$, where the maximum is taken for all clean complete markings of $M$, where $M'$ differs from $M$ by an \textit{elementary move}. Let $l_2=\textnormal{max}_G \textnormal{min}_M i(G,M)$, where the minimum is taken over graphs with at most $k=\textnormal{max}(k_1,k_2)$ edges that fill $S$ and the maximum is taken over clean complete markings of $S$.

We then have $\mathcal{M}_{k,l}(S)$ connected, where $l=\textnormal{max}(l_1,l_2)$. Endow the graph $\mathcal{M}_{k,l}(S)$ with a metric where each edge has unit length, and distance is given by shortest paths. Vertex stabilizers are uniformally bounded, by the Alexander method. The \textit{mapping class group} $\mathcal{MCG}(S)$ acts on $\mathcal{M}_{k,l}(S)$, and thus by the Milnor-\v{S}varc Lemma \cite[Proposition I.8.19]{BridsonHaefliger} the mapping class group is quasi-isometric to the \textit{marking graph} $\mathcal{M}_{k,l}(S)$.

\begin{thm}
\label{bm}

Suppose a multicurve $\alpha$ and a geodesic $g=(\gamma_i)$ satisfy $\gamma_i,\alpha$ fill $S$ for each $i$. Then $\textnormal{diam}_{\mathcal{M}_{k,l}(S)}(\Gamma_\alpha(g))\leq M$, where $M$ depends on $S$.

\end{thm}

\begin{proof}

We sketch a proof for brevity, since most of the proof is a generalization of earlier lemmas. Firstly, if $\Gamma_1$ intersects $\Gamma_2$ boundedly many times, then there are only finitely many possibilities for $\Gamma_2$ in terms of $\Gamma_1$. There are only finitely many possibilities for $\Gamma_1$ modulo homeomorphism. Thus, if intersection between markings is bounded then their distance is bounded.

Secondly, one bounds $i(\Gamma_\alpha(\gamma_1),\Gamma_\alpha(\gamma_2))$ when $\gamma_1$ and $\gamma_2$ are disjoint, in terms of $S$. This generalizes Lemma \ref{bp}. Then one bounds $i(\Gamma_\alpha(\beta),\Gamma_\alpha(\gamma))$ when $\gamma$ is a $(\alpha,\beta)$-loop, in terms of $S$. This generalizes Lemma \ref{cirproj}. Using these lemmas, one can generalize Lemma \ref{bg} then finish the argument analogously to Theorem \ref{bgit}.\end{proof}

\subsection*{Acknowledgements} The author would like to thank Saul Schleimer for thorough comments on the paper. We thank Brian Bowditch, Saul Schleimer and Robert Tang for interesting conversations.

\bibliography{referencesbgit}

\begin{bibdiv}
\begin{biblist}

\bib{Aougab13}{article}{
      author={Aougab, T.},
       title={Uniform hyperbolicity of the graphs of curves},
        note={\url{http://arxiv.org/abs/1212.3160}},
}

\bib{Bow13}{article}{
      author={Bowditch, B.H.},
       title={Uniform hyperbolicity of the curve graphs},
  note={\url{http://homepages.warwick.ac.uk/~masgak/papers/uniformhyp.pdf}},
}

\bib{BridsonHaefliger}{book}{
      author={Bridson, Martin~R.},
      author={Haefliger, Andr{\'e}},
       title={Metric spaces of non-positive curvature},
      series={Grundlehren der Mathematischen Wissenschaften [Fundamental
  Principles of Mathematical Sciences]},
   publisher={Springer-Verlag},
     address={Berlin},
        date={1999},
      volume={319},
        ISBN={3-540-64324-9},
      review={\MR{1744486 (2000k:53038)}},
}

\bib{ClayRafiSchleimer13}{article}{
      author={Clay, M.T.},
      author={Rafi, K.},
      author={Schleimer, S.},
       title={Uniform hyperbolicity of the curve graph via surgery sequences},
        note={in preparation},
}

\bib{FarbMargalit}{book}{
      author={Farb, Benson},
      author={Margalit, Dan},
       title={A primer on mapping class groups},
      series={Princeton Mathematical Series},
   publisher={Princeton University Press},
     address={Princeton, NJ},
        date={2012},
      volume={49},
        ISBN={978-0-691-14794-9},
      review={\MR{2850125 (2012h:57032)}},
}

\bib{Harvey}{incollection}{
      author={Harvey, W.~J.},
       title={Boundary structure of the modular group},
        date={1981},
   booktitle={Riemann surfaces and related topics: {P}roceedings of the 1978
  {S}tony {B}rook {C}onference ({S}tate {U}niv. {N}ew {Y}ork, {S}tony {B}rook,
  {N}.{Y}., 1978)},
      series={Ann. of Math. Stud.},
      volume={97},
   publisher={Princeton Univ. Press},
     address={Princeton, N.J.},
       pages={245\ndash 251},
      review={\MR{624817 (83d:32022)}},
}

\bib{HenselPrzytyckiWebb13}{article}{
      author={Hensel, S.},
      author={Przytycki, P.},
      author={Webb, R. C.~H.},
       title={Slim unicorns and uniform hyperbolicity for arc graphs and curve
  graphs},
        note={\url{http://arxiv.org/abs/1301.5577}},
}

\bib{LeasureThesis}{book}{
      author={Leasure, Jason~Paige},
       title={Geodesics in the complex of curves of a surface},
   publisher={ProQuest LLC, Ann Arbor, MI},
        date={2002},
        ISBN={978-0496-62255-9},
  url={http://gateway.proquest.com/openurl?url_ver=Z39.88-2004&rft_val_fmt=info:ofi/fmt:kev:mtx:dissertation&res_dat=xri:pqdiss&rft_dat=xri:pqdiss:3114766},
        note={Thesis (Ph.D.)--The University of Texas at Austin},
      review={\MR{2705485}},
}

\bib{MasurMinsky00}{article}{
      author={Masur, H.~A.},
      author={Minsky, Y.~N.},
       title={Geometry of the complex of curves. {II}. {H}ierarchical
  structure},
        date={2000},
        ISSN={1016-443X},
     journal={Geom. Funct. Anal.},
      volume={10},
      number={4},
       pages={902\ndash 974},
         url={http://dx.doi.org/10.1007/PL00001643},
      review={\MR{1791145 (2001k:57020)}},
}

\bib{MasurMosherSchleimer}{article}{
      author={Masur, Howard},
      author={Mosher, Lee},
      author={Schleimer, Saul},
       title={On train-track splitting sequences},
        date={2012},
        ISSN={0012-7094},
     journal={Duke Math. J.},
      volume={161},
      number={9},
       pages={1613\ndash 1656},
  url={http://0-dx.doi.org.pugwash.lib.warwick.ac.uk/10.1215/00127094-1593344},
      review={\MR{2942790}},
}

\bib{MasurSchleimer13}{article}{
      author={Masur, Howard},
      author={Schleimer, Saul},
       title={The geometry of the disk complex},
        date={2013},
        ISSN={0894-0347},
     journal={J. Amer. Math. Soc.},
      volume={26},
      number={1},
       pages={1\ndash 62},
         url={http://dx.doi.org/10.1090/S0894-0347-2012-00742-5},
      review={\MR{2983005}},
}

\bib{MasurMinsky99}{article}{
      author={Masur, Howard~A.},
      author={Minsky, Yair~N.},
       title={Geometry of the complex of curves. {I}. {H}yperbolicity},
        date={1999},
        ISSN={0020-9910},
     journal={Invent. Math.},
      volume={138},
      number={1},
       pages={103\ndash 149},
         url={http://dx.doi.org/10.1007/s002220050343},
      review={\MR{1714338 (2000i:57027)}},
}

\end{biblist}
\end{bibdiv}
\bibliographystyle{amsalpha}

\end{document}